\documentclass[12pt,a4paper,oneside]{amsart}
\usepackage[utf8]{inputenc}
\usepackage{amsmath,amscd, amsfonts, amssymb, amsthm}
\usepackage{enumerate,enumitem}
\usepackage{setspace,kantlipsum}
\usepackage{tikz-cd}
\usepackage{mathrsfs}
\usepackage{textcomp}
\usepackage[margin=1in]{geometry}
\usepackage{colonequals}
\usepackage{mathtools}
\usepackage{sansmath}
\usepackage[colorlinks=true, linkcolor=blue, citecolor=blue]{hyperref}

\theoremstyle{plain} 
\newtheorem{thm}{Theorem}[section]
\newtheorem{lemma}[thm]{Lemma}
\newtheorem{lem}[thm]{Lemma}
\newtheorem{proposition}[thm]{Proposition}
\newtheorem{prop}[thm]{Proposition}
\newtheorem{corollary}[thm]{Corollary}

\theoremstyle{definition}

\newtheorem{definition}[thm]{Definition}

\newtheorem{remark}[thm]{Remark}

\numberwithin{equation}{section}
\newcommand{\C}{\mathbb{C}}
\newcommand{\Z}{\mathbb{Z}}
\newcommand{\Q}{\mathbb{Q}}

\newcommand{\N}{\mathbb{N}}

\newcommand{\cH}{\mathcal{H}}
\newcommand{\cI}{\mathcal{I}}
\newcommand{\cJ}{\mathcal{J}}

\newcommand{\cL}{\mathcal{L}}
\newcommand{\cM}{\mathcal{M}}
\newcommand{\cN}{\mathcal{N}}
\newcommand{\cO}{\mathcal{O}}

\newcommand{\sD}{\mathscr{D}}

\newcommand{\mc}[1]{{\mathcal{#1}}}
\newcommand{\mb}[1]{{\mathbb{#1}}}
\newcommand{\ms}[1]{{\mathscr{#1}}}
\newcommand{\mrm}[1]{{\mathrm{#1}}}
\newcommand{\mbf}[1]{{\mathbf{#1}}}

\renewcommand{\d}{\partial}

\newcommand{\gr}{\mathrm{gr}}

\newcommand{\DR}{\mathrm{DR}}

\newcommand{\MHM}{\mathrm{MHM}}
\let\mhm\MHM

\newcommand{\qcoh}{\mathrm{QCoh}}
\newcommand{\Lotimes}{\mathbin{\overset{\mathbf{L}}\otimes}}

\DeclareMathOperator{\spec}{Spec}

\DeclareMathOperator{\coker}{coker}

\begin{document}

\title{On the Hodge and $V$-filtrations of mixed Hodge modules}

\author{Dougal Davis and Ruijie Yang}
\thanks{DD was supported by the ARC grant FL200100141.}
\begin{abstract}
In this paper, we prove a Beilinson-type formula for the $V$-filtration of Kashiwara and Malgrange on a complex mixed Hodge module, using Hodge filtrations on the localization. Our formula expresses the $V$-filtration as the filtered $\sD$-module underlying a pro-mixed Hodge module.

We apply this to the theory of higher multiplier and Hodge ideals. Our first result shows that higher multiplier ideals can be obtained directly from Hodge ideals by taking a suitable limit. As a corollary, we deduce that Hodge ideals are left semi-continuous if and only if they coincide with higher multiplier ideals, thereby improving results of Saito and Musta\c{t}\u{a}--Popa and resolving a folklore question. We further prove a birational transformation formula for higher multiplier ideals, generalising the classical formula for multiplier ideals and answering a question of Schnell and the second author. Finally, we provide very quick proofs of the main vanishing theorems for Hodge ideals, and strengthen a result of B. Chen.
\end{abstract}

\maketitle

\section{Introduction}

In this note, we give a general expression for the $V$-filtration of Kashiwara and Malgrange in terms of the Hodge filtration in Saito's theory of mixed Hodge modules. We give illustrative applications to the theory of higher multiplier ideals and Hodge ideals in birational geometry, including a precise comparison between the two. We also provide a birational transformation formula for higher multiplier ideals and give quick proofs of the main vanishing theorems of Hodge ideals.

Saito's theory of mixed Hodge modules \cite{Saito88, Saito90} extends the theory of variations of Hodge structure to a six functor formalism in the sense of Grothendieck. A mixed Hodge module in Saito's sense consists of a regular holonomic $\ms{D}$-module $\mc{M}$ equipped with a Hodge filtration, a weight filtration and a $\mb{Q}$-structure, satisfying some rather stringent conditions. A central point in the theory is to specify how these data should behave under direct image functors, the most subtle case being the extension of the Hodge filtration along an open immersion. This is controlled in Saito's theory using an auxiliary ingredient: the \emph{Kashiwara-Malgrange $V$-filtration}, an important object in $\ms{D}$-module theory in its own right.

Conversely, we show that $V$-filtrations on mixed Hodge modules can be recovered from the six functor formalism, or more precisely, from the behaviour of Hodge filtrations under open immersions. This perspective is powerful, as it allows one to treat $V$-filtrations as filtered $\sD$-modules underlying pro-mixed Hodge modules (in the sense of \cite[\S 2.2]{DMB}, for example), which leads to the applications discussed above.  While these applications can also be proved by other means, we show in a forthcoming paper \cite{DY} that the methods introduced here can be pushed much further and use them to prove the Strong Monodromy Conjecture for hyperplane arrangements.

Our result is most conveniently stated in the language of \emph{complex} mixed Hodge modules of Sabbah and Schnell \cite{MHMproject}, a generalisation of Saito's theory without $\mb{Q}$-structures. From now on, by ``mixed Hodge modules'' we will always mean complex mixed Hodge modules. We refer the reader who would prefer to work directly in Saito's theory to \S\ref{subsec:foundations} below. Our main theorem is as follows. Let $X$ be a complex algebraic variety\footnote{All of our results can, with appropriate care, be extended to the setting of complex manifolds and holomorphic functions: we leave the details to the interested reader (see also \cite{DY}).} and $f : X \to \mb{C}$ a regular function. Denote by $j:U=f^{-1}(\C^\ast)\to X$ the open embedding. For a holonomic (left) $\ms{D}_X$-module $\mc{M}$ on which $f$ acts bijectively, we may form the $\ms{D}_X[s]$-module $\mc{M}[s]f^s$ on the one hand and the pushforward $\iota_+\mc{M}$ under the graph embedding $\iota : X \to X \times \mb{C}$ on the other hand. It was shown by Malgrange \cite{Malgrange} (see also \cite{MPVfiltration}) that there is an isomorphism of $\ms{D}_X\langle s, t\rangle$-modules
\begin{equation}\label{eqn: Malgrange iso introduction}
    \rho:\cM[s]f^s \xrightarrow{\sim} \iota_{+}\cM,
\end{equation} 
where the coordinate $t$ on $\mb{C}$ acts on the left by $s \mapsto s + 1$ and the formal variable $s$ acts on the right by $-\partial_t t$. Our main result refines this when $\mc{M}$ is a mixed Hodge module.

\begin{thm}\label{thm: V filtration via power of fs}
    Assume that $\cM$ underlies a mixed Hodge module. Then for $\alpha \in \mb{R}$ and each $p\in \Z$, the inverse systems
    \[ \left\{F_p j_{\ast}\left(\frac{j^{\ast}\cM[s]f^s}{(s+\alpha)^n} \right)\right\}_{n\geq 0} \quad \text{and} \quad \left\{F_p j_!\left(\frac{j^{\ast}\cM[s]f^s}{(s+\alpha)^n} \right)\right\}_{n\geq 0}\]
    are constant for $n\gg 0$, where $F_\bullet$ denotes the Hodge filtration on the $!$ and $*$ extensions of the mixed Hodge modules $j^*\mc{M}[s]f^s/(s + \alpha)^n$. Moreover, \eqref{eqn: Malgrange iso introduction} induces isomorphisms of $\sD_X[s]$-modules
    \begin{gather} \label{eq:magic formulas}
    \begin{aligned}
    \bigcup_p \varprojlim_{n} F_p j_{\ast}\left(\frac{j^{\ast}\cM[s]f^s}{(s+\alpha)^n} \right)&\xrightarrow{\sim }V^{\alpha}\iota_{+}\cM,\\
    \bigcup_p \varprojlim_{n} F_p j_{!}\left(\frac{j^{\ast}\cM[s]f^s}{(s+\alpha)^n} \right)&\xrightarrow{\sim }V^{>\alpha}\iota_{+}\cM.
    \end{aligned}
    \end{gather}
    If $\alpha \geq 0$ then these isomorphisms respect the Hodge filtrations up to an index shift by $1$: 
    \[  \varprojlim_{n} F_p j_{\ast}\left(\frac{j^{\ast}\cM[s]f^s}{(s+\alpha)^n} \right)\xrightarrow{\sim} F_{p+1}V^{\alpha}\iota_{+}\cM,\]
    and similarly for $j_!$ and $V^{>\alpha}$.
\end{thm}

Here we note that $j^{\ast}\cM[s]f^s/(s+\alpha)^n$ is the tensor product of $j^*\mc{M}$ with an admissible variation of mixed Hodge structure on $U$, so it naturally inherits the structure of a mixed Hodge module. 

For an arbitrary holonomic $\sD$-module $\cM$ on which $f$ acts bijectively, there is a well-known isomorphism 
    \begin{equation}\label{eqn: Beilinson}  \gr^{\alpha}_V\iota_{+}\cM\xrightarrow{\sim} \textrm{coker}\left(\varprojlim_n j_{!}\left(\frac{j^{\ast}\cM[s]f^s}{(s+\alpha)^n}\right) \to \varprojlim_n j_{\ast}\left(\frac{j^{\ast}\cM[s]f^s}{(s+\alpha)^n}\right)\right). \end{equation}
The right hand side is Beilinson's expression for the nearby cycles functor \cite{Beilinson}, while the isomorphism is essentially due to Kashiwara \cite{Kas83} (see also \cite{MM} and \cite[Lemma 6.2]{DV22}). Since $\bigcup_p \varprojlim_n F_p$ is nothing but the inverse limit in the category of exhaustively filtered $\ms{D}$-modules, one can think of Theorem \ref{thm: V filtration via power of fs} as a lift of \eqref{eqn: Beilinson} to $V^{\alpha}\iota_{+}\cM$ and $V^{>\alpha}\iota_+\mc{M}$ when $\cM$ underlies a mixed Hodge module. 

We remark that the $\ms{D}$-modules on the left hand side of \eqref{eq:magic formulas} were used to establish deformation and wall-crossing properties of Hodge filtrations in \cite[\S 3.3]{DV23} (see also \cite[\S 2]{DMB}). This note arose from an attempt to relate these to more classical constructions in $\ms{D}$-module theory.

Consider the following map induced by the evaluation map and \eqref{eqn: Malgrange iso introduction}:
\begin{equation}\label{eqn: evaulation map for any M}ev_{s=-\alpha}:\iota_{+}\cM\xrightarrow{\rho^{-1}} \cM[s]f^s\xrightarrow{s=-\alpha} \cM\cdot f^{-\alpha}.\end{equation}
By construction, we have short exact sequences of mixed Hodge modules
\[ 0 \to j_*\left(\frac{j^*\mc{M}[s]}{(s + \alpha)^{n - 1}}\right)(1) \xrightarrow{s + \alpha} j_*\left(\frac{j^*\mc{M}[s]}{(s + \alpha)^n}\right) \to j_*(j^{\ast}\mc{M} \cdot f^{-\alpha}) \to 0 \]
for all $n$, and similarly for $j_!$. Hence, Theorem \ref{thm: V filtration via power of fs} implies the following corollary.

\begin{corollary}\label{corollary: ses of Hodge on V}
 For any $\cM$ as above and any $\alpha \geq 0$, the morphism
 \[ s + \alpha : V^\alpha \iota_+\mc{M} \to V^\alpha \iota_+\mc{M} \{-1\} \]
 is strict with respect to the Hodge filtrations, where $F_k\cN\{-1\}=F_{k+1}\cN$.
 Moreover, for $k \in \mb{Z}$, we have short exact sequences
    \begin{align*}
        0 \to F_{k}V^\alpha \iota_+ \cM \xrightarrow{s+\alpha} &F_{k+1} V^\alpha \iota_+ \cM \xrightarrow{ev_{s=-\alpha}} F_k(\cM \cdot f^{-\alpha}) \to 0, \\
    0 \to F_{k}V^{>\alpha} \iota_+ \cM \xrightarrow{s+\alpha} &F_{k+1} V^{>\alpha} \iota_+ \cM \xrightarrow{ev_{s=-\alpha}} F_k(j_{!}(j^{\ast}\cM \cdot f^{-\alpha})) \to 0
    \end{align*}
    and hence a strict filtered surjection
    \begin{equation}\label{eqn: surjection from V to j!ast}
    V^{>\alpha}\iota_{+}\cM\{-1\} \twoheadrightarrow j_{!\ast}(j^{\ast}\cM\cdot f^{-\alpha}).
    \end{equation}
\end{corollary}

Having formulated Corollary \ref{corollary: ses of Hodge on V}, one can also prove it directly without appealing to Theorem \ref{thm: V filtration via power of fs}. If $\cM$ is a simple regular holonomic $\sD$-module, \eqref{eqn: surjection from V to j!ast} is proved (without Hodge filtrations) using a $b$-function argument in \cite[Proposition 5.10]{DLY}. In general, the injectivity of $s + \alpha$ is elementary and the surjectivity of the evaluation maps $ev_{s=-\alpha}$ in Corollary \ref{corollary: ses of Hodge on V} can also be proved using Saito's formula for the Hodge filtrations on $\ast$- and $!$-extensions (see \S \ref{sec: extension of Hodge filtration}). For the surjectivity of $ev_{s=-\alpha}:F_{k+1}V^{\alpha}\iota_{+}\cM\to F_k(\cM\cdot f^{-\alpha})$, see \cite[Theorem 1.6]{dakin} or \cite[Proposition 3.1]{MingyiZhang}. 

\subsection{Higher multiplier ideals as limits of Hodge ideals}
Now let $D$ be an effective divisor on a smooth variety $X$ and $\alpha>0$. The \emph{Hodge ideals} $I_k(\alpha D)$ \cite{MPbirational} and \emph{higher multiplier ideals} $\cI_{k,-\alpha}(D)$ \cite{SY23} (see also \cite{Saito16}) are two families of ideals introduced recently, which both recover the classical multiplier ideal $\cJ((\alpha-\epsilon)D)$ if $k=0$. They are known to encode similar information about the singularities of $D$: for example, both can be used to determine Saito's minimal exponent in the same way that the multiplier ideals determine the log canonical threshold \cite{MPVfiltration,SY23}.

Suppose locally that $D=\mathrm{div}(f)$ and write $\tilde{I}_k(\alpha D)\colonequals \cI_{k,-\alpha}(D)$. By definition,
\[ \tilde{I}_k(\alpha D)\otimes \d_t^k \cong \gr^F_{k+1}V^{\alpha}\iota_{+}\cO_X.\]
On the other hand, in Lemma \ref{lemma: two definitions of Hodge ideals agree}, we prove that the Hodge ideals can be defined by
\begin{equation}\label{eqn: new point of view of Hodge ideals}
I_k(\alpha D)\otimes \cO_X(kD)\cdot f^{-\alpha}=F_kj_{\ast}(\cO_U\cdot f^{-\alpha}).
\end{equation}
Then Corollary \ref{corollary: ses of Hodge on V} gives
    \begin{equation}\label{eqn: Hodge ideal via Vfiltration}I_k(\alpha D)\otimes \cO_X(kD)\cdot f^{-\alpha}=ev_{s=-\alpha}(F_{k+1}V^{\alpha}\iota_{+}\cO_X),\end{equation}
which recovers the main result of Musta\c{t}\u{a} and Popa \cite[Theorem A]{MPVfiltration} and a result of Mingyi Zhang \cite[Proposition 3.1]{MingyiZhang}. Corollary \ref{corollary: ses of Hodge on V} also implies a weighted version of \eqref{eqn: Hodge ideal via Vfiltration}, which is obtained independently by Dakin \cite[Theorem 1.6]{dakin}. Note that the equation \eqref{eqn: Hodge ideal via Vfiltration} implies that \begin{equation*}\label{eqn: comparison modulo I} I_k(\alpha D)=\tilde{I}_k(\alpha D) \quad \textrm{mod $\cI_D$},\end{equation*}
which was proved by Saito \cite{Saito16} ($\alpha=1$) and  \cite[Theorem A$'$]{MPVfiltration} ($\alpha\in \Q$).

However, these two ideals are not equal in general. For example, let $D=(x^2+y^3=0)$ and  $5/6\leq \alpha<1$, then \cite[Example 10.5]{MPbirational} and \cite[Proposition 4.3]{MingyiZhang} (see also Example 6.20 in the first \texttt{arXiv} version of \cite{SY23} via the Thom-Sebastiani-type formula for higher multiplier ideals) give
\begin{equation}\label{eqn: Hodge and higher are not equal for cusp}I_{2}(\alpha D)=(x^3,x^2y^2,xy^3,y^4-(2\alpha+1)x^2y)\neq (x^3,x^2y^2,xy^3,y^5,x^2y)=\tilde{I}_2(\alpha D).\end{equation}
Finding the precise relation between the two ideals remained a folklore question. We solve this mystery by showing that one can recover the higher multiplier ideal from the set of all Hodge ideals
as follows. We say $\cI\subseteq\cO_{X\times S}$ is \emph{a family of ideals over a scheme $S$} if the quotient $\cO_{X\times S}/\cI$ is flat over $S$. 
\begin{thm}\label{thm:higher vs hodge}
For each $\alpha > 0$ and all $k\in \N$, there exists a unique family of ideals $\mathfrak{I}_{k,\alpha} \subseteq \mc{O}_{X\times \mb{P}^1}$ over $\mb{P}^1$ such that
\[ I_k(\beta D) = \mathfrak{I}_{k,\alpha}|_{X \times \{\beta\}} \quad \text{if $\beta=\alpha-\epsilon, \quad0 \leq \epsilon \ll 1$}.\]
The higher multiplier ideal $\tilde{I}_k(\alpha D)$ is given by evaluating this family at infinity:
\[ \tilde{I}_k(\alpha D) = \mathfrak{I}_{k,\alpha}|_{X \times \{\infty\}}.\] 
\end{thm}
To illustrate, consider $D=(x^2+y^3=0)$ and $\alpha\in [5/6,1)$. Then $\mathfrak{I}_{k,\alpha}$ is the algebraic family defined by
\[ \mathfrak{I}_{k,\alpha}|_{\mb{C}^2 \times \{\beta\}} = (x^3, x^2y^2, xy^3, y^4 - (2\beta + 1)x^2 y), \quad \text{for $\beta \in \mb{C}$}.\]
Taking the limit as $\beta \to \infty$ in the Hilbert scheme of finite length ideals in $\mb{C}^2$, we get
\begin{align*}
\mathfrak{I}_{k,\alpha}|_{\mb{C}^2 \times \{\infty\}} &= \lim_{\beta \to \infty} (x^3, x^2y^2, xy^3, y^4 - (2\beta + 1)x^2 y) \\
&= \lim_{\beta \to \infty}\left(x^3, x^2y^2, xy^3, y^5, \frac{1}{2\beta + 1}y^4 -  x^2 y\right)\\
&= (x^3, x^2y^2, xy^3, y^5, x^2y) \\
&= \tilde{I}_2(\alpha D).
\end{align*}

It is known that the higher multiplier ideals $\tilde{I}_{k}(\alpha D)$ are decreasing in $\alpha$ and left continuous: $\tilde{I}_{k}(\alpha D) \supseteq \tilde{I}_{k}(\beta D)$, whenever $\alpha \leq \beta$, and $\tilde{I}_{k}(\alpha D)=\tilde{I}_{k}((\alpha-\epsilon)D)$ for $0 < \epsilon \ll 1$; \cite[Proposition 3.7 and \S 1.7]{SY23}. These are not true for the Hodge ideals, i.e.\ they are incommensurable and not left continuous (see \eqref{eqn: Hodge and higher are not equal for cusp} for example). This gives another reason why $I_k(\alpha D)$ and $\tilde{I}_k(\alpha D)$ cannot be equal in general. However, Theorem \ref{thm:higher vs hodge} implies that this is the only obstruction.
\begin{corollary}\label{cor: obstruction to left continuity}
We have $\tilde{I}_k(\alpha D) = I_k(\alpha D)$
    if and only if  the Hodge ideals are left continuous at $\alpha$, i.e.
    \[I_k(\alpha D) = I_k((\alpha - \epsilon) D), \quad \textrm{for $0 < \epsilon \ll 1$}.\]
\end{corollary}
This relates the two features of the example \eqref{eqn: Hodge and higher are not equal for cusp}: $I_2(\alpha D)\neq \tilde{I}_2(\alpha D)$ because $I_2(\alpha D)$ is not left continuous for $\alpha\in [5/6,1)$. The example $D=(x^3+y^3+z^3=0)\subseteq \C^3$ in \cite[Remark 9.8]{PopaICM} where $I_2(D)\neq \tilde{I}_2(D)$ can be explained in a similar fashion. These results allow us to obtain new computations of higher multiplier ideals using the corresponding Hodge ideals. For example, \cite[Example 11.7]{MPbirational} implies the following.
\begin{corollary}
If $x\in D$ is an ordinary singularity of multiplicity $m$ and $\dim X=km+r$ for a unique $0\leq r\leq m-1$, then
\[ \tilde{I}_k\left(\frac{j+r}{m}\cdot D\right)_x=\mathfrak{m}_x^{j}, \quad \textrm{if $0\leq j\leq \min(m-r,m-1)$}.\]
\end{corollary}

\subsection{Vanishing theorems}

In applications to algebro-geometric questions, Nadel-type vanishing theorems for higher multiplier and Hodge ideals play a crucial role, just as in the case of classical multiplier ideals. However, establishing these theorems is considerably more subtle than in the classical setting. For Hodge ideals, Musta\c{t}\u{a}--Popa first proved such results in \cite[Theorem F]{MPHodgeideal} for $\Z$-divisors, and later extended them to $\Q$-divisors in \cite[Theorem 12.1]{MPbirational}, under the additional assumption that the line bundle $\cO_X(\ell D)$ admits $\ell$-th roots. They conjectured that this assumption is unnecessary. This conjecture was subsequently proved by B. Chen, who established the following result and deduced from it the most general vanishing theorem for Hodge ideals \cite[Theorem 1.1]{Chenbingyi}. Let $X$ be a smooth projective variety of dimension $n$, and $B,D$ are effective divisors on $X$. 

\begin{thm}[{\cite[Theorem 1.2(1)]{Chenbingyi}}] \label{thm: Bingyi Chen}
If $B-\alpha D$ is ample, then
    \[ \mb{H}^i(X,\mathrm{Sp}_k(\cM_{\bullet}(\alpha D))\otimes_{\cO_X}\cO_X(B))=0,\quad \textrm{whenever $i>0$}.\]
\end{thm}
Here $\mathrm{Sp}_k(\cM_{\bullet}(\alpha D))$ denotes the complex 
\begin{equation}\label{eqn: spencer complex} [\cM_k(\alpha D)\otimes \wedge^n T_X \to \cM_{k+1}(\alpha D)\otimes \wedge^{n-1} T_X \to \cdots \cM_{k+n}(\alpha D)\otimes \wedge^0 T_X],
\end{equation}
where
\[ \cM_k(\alpha D) \colonequals \frac{\omega_X(kD)\otimes I_k(\alpha D)}{\omega_X((k-1)D)\otimes I_{k-1}(\alpha D)}.\]
The proof of these vanishing theorems relies on a careful analysis of the sheaves of logarithmic forms associated to $D$ and $B$.

In \S \ref{vanishing theorem}, we provide a much shorter proof of Theorem \ref{thm: Bingyi Chen}, using our characterization of Hodge ideals in \eqref{eqn: new point of view of Hodge ideals} together with the Kodaira vanishing theorem for twisted Hodge modules \cite{SY23,DV23}. Moreover, we strengthen Chen’s result by proving the following theorem.
\begin{thm}\label{thm: strengthen of Bingyi}
Suppose that either:
\begin{enumerate}
\item The tangent bundle of $T_X$ is trivial and $B - \alpha D$ is ample,
\item $X$ is a homogeneous variety and $B -\alpha D + \frac{1}{2}K_X$ is semi-ample, or
\item $X = \mb{P}^n$ and $B - \alpha D$ is ample.
\end{enumerate}
Then
\[ \mb{H}^i(X, \cM_k(\alpha D)\otimes_{\mc{O}_X}\mc{O}_X(B)) = 0, \quad \text{whenever $i > 0$}.\]
\end{thm}

\subsection{Functoriality of higher multiplier ideals}
In \cite{SY23}, Schnell and the second author asked how higher multiplier ideals behave under birational transformations. We answer this question in the present paper. Let $D$ be an effective divisor on $X$, and let $\pi:\tilde{X}\to X$ be a log resolution of $(X,D)$ such that $\pi$ is an isomorphism over $X\setminus D$ and $\pi^{\ast}D$ has simple normal crossings. Abusing notation, write $\iota$ for the total graph embeddings of both $D$ and $\pi^{\ast}D$. Then Theorem \ref{thm: V filtration via power of fs} implies that
\begin{equation*}
\bigoplus_k \gr^F_{k + 1} V^\alpha\iota_+\mc{O}_X = \mbf{R}\pi_*\left(\omega_{\tilde X/X} \otimes \mrm{Sym}(\pi^*T_X) \Lotimes_{\mrm{Sym}(T_{\tilde X})}\left(\bigoplus_k \gr^F_{k + 1} V^\alpha\iota_+\mc{O}_{\tilde X}\right)\right),
\end{equation*}
where $\pi^*T_X$ and $T_{\tilde X}$ have graded degree $1$. This gives an immediate answer to this question.
\begin{proposition}\label{prop: containment of higher multiplier ideal}
We have $ \pi_*(\omega_{\tilde X/X} \otimes \tilde{I}_{k}(\alpha \pi^{\ast}D)) \subseteq \tilde{I}_{k}(\alpha D)$.
\end{proposition}
Note that for $k=0$ this is an equality, recovering the classical birational formula for multiplier ideals, while for $k\geq 1$ equality need not hold. The proof is given in \S \ref{sec: functoriality higher multiplier ideal}. Moreover, the formula above provides, in principle, a solution to another question in \cite{SY23}, namely a (derived) birational transformation formula for higher multiplier ideals. This can be made more explicit using a recent result of Chen–Musta\c{t}\u{a} \cite[Theorem 4.1]{chenmustata}, who construct an explicit filtered resolution of $(V^{\alpha}\iota_{+}\cO_X,F_{\bullet})$ in the normal crossings case.

\subsection{Remarks on the foundations} \label{subsec:foundations}

In this paper, we have made free use of the theory of complex mixed Hodge modules. Since the foundations for this theory \cite{MHMproject} are unfinished at the time of writing, we briefly describe how to prove the same results in the framework of Saito \cite{Saito88, Saito90}. 

Let $X$ be a smooth variety and write $\mhm_{\mb{Q}}(X)$ for Saito's $\mb{Q}$-linear category of mixed Hodge modules on $X$. If $K \subseteq \mb{C}$ is any algebraic field extension of $\mb{Q}$, we may formally define a $K$-linear category $\mhm_K(X)$ of $K$-mixed Hodge modules as follows. First, suppose that $K/\mb{Q}$ is finite. Since $\mhm_{\mb{Q}}(X)$ is tensored over \emph{finite} dimensional $\mb{Q}$-vector spaces, we can define $\mhm_K(X) = \mrm{Mod}_K( \mhm_{\mb{Q}}(X))$ to be the category of $K$-modules in this category, i.e., objects $M \in \mhm_{\mb{Q}}(X)$ equipped with an action $K \otimes M \to M$ satisfying the usual axioms for a module over a ring. Note that this definition makes sense even if $K \not\subset \mb{R}$. In general, if $K$ is not a finite extension (e.g.\ $K = \bar{\mb{Q}}$), form the category $\operatorname{Ind} \mhm_{\mb{Q}}(X)$ of ind-objects: this category is tensored over the category of $\mb{Q}$-vector spaces of \emph{arbitrary} dimension. We let
\[ \mhm_K(X) \subseteq \mrm{Mod}_K(\operatorname{Ind} \mhm_{\mb{Q}}(X))\]
be the smallest abelian subcategory containing the free objects $K \otimes_{\mb{Q}} M$ for $M \in \mhm_{\mb{Q}}(X)$. Note that any object $M_K$ in $\mhm_K(X)$ is of the form $K \otimes_{K'} M_{K'}$ for some finite extension $K'/\mb{Q}$ and some $M_{K'} \in \mhm_{K'}(X)$.

Since it is defined by a formal categorical construction, the category $\mhm_K(X)$ inherits the six operations from those on $\mhm_\mb{Q}(X)$. Moreover, we have a ``forgetful functor'' from $\mhm_K(X)$ to the category of filtered regular holonomic $\ms{D}_X$-modules defined as follows. Every $M \in \mhm_K(X)$ has an underlying ind-filtered $\ms{D}_X$-module $(\mc{M}^{\mathit{big}}, F_\bullet)$, equipped with an action of the algebra $K \otimes_{\mb{Q}} \mb{C}$. We set
\[ (\mc{M}, F_\bullet\mc{M}) = (\mb{C} \otimes_{K \otimes_{\mb{Q}}\mb{C}} \mc{M}^{\mathit{big}}, \mb{C} \otimes_{K \otimes_{\mb{Q}}\mb{C}} F_\bullet\mc{M}^{\mathit{big}}).\]
When $K/\mb{Q}$ is finite, this is a direct summand of $(\mc{M}^{\mathit{big}}, F_\bullet)$. By reduction to this case, it is easy to see that $(\mc{M}, F_\bullet)$ is a filtered regular holonomic $\ms{D}$-module and that the functor $M \mapsto (\mc{M}, F_\bullet)$ is exact and faithful, commutes with proper pushforwards etc.

All the results of this paper hold with the $\bar{\mb{Q}}$-linear category $\mhm_{\bar{\mb{Q}}}(X)$ as a stand-in for the category $\mhm(X)$ of complex mixed Hodge modules, with the caveat that we must restrict to $\alpha \in \mb{Q}$ in Theorem \ref{thm: V filtration via power of fs} and its consequences (this is no great loss, as all objects in $\mhm_{\bar{\mb{Q}}}(X)$ have $V$-filtrations indexed by $\mb{Q}$). Indeed, the only property we need from complex mixed Hodge modules that does not already hold in $\mhm_{\mb{Q}}$ is that the local system $\mc{O}_{\mb{C}^*}t^\alpha$ should lift to an object in $\mhm(\mb{C}^*)$ for all $\alpha \in \mb{R}$. For $\alpha = n/d \in \mb{Q}$, we have such a lift in $\mhm_{\bar{\mb{Q}}}(\mb{C}^*)$, given by the eigenspace
\[ \pi_*(\bar{\mb{Q}}^H[1])_{T = e^{-2\pi i \alpha}},\]
where $\pi \colon \mb{C}^* \to \mb{C}^*$ is $z \mapsto z^d$ and $T$ is the automorphism induced by the deck transformation $z \mapsto e^{-2\pi i/d}z$. This idea is similar to \cite[\S 2]{MPbirational}.

\subsection*{Acknowledgements}

We thank Christian Schnell and Mircea Musta\c{t}\u{a} for useful discussions, and especially Bradley Dirks for detailed comments of the manuscript. R.Y. would like to thank Christian Sevenheck for questions which eventually led to Theorem \ref{thm:higher vs hodge}. This project began when both authors were visiting the Simons Center for Geometry and Physics at Stony Brook during the 2024 winter school ``New applications of mixed Hodge modules''. We would like to thank Christian Schnell and Bradley Dirks for organising this excellent event and for the invitation to attend.

\section{A Beilinson-type formula for the $V$-filtration} \label{sec:hypersurface V filtration}

\subsection{The $V$-filtration along a smooth divisor}

In this subsection, we recall the definition of the $V$-filtration of a holonomic $\ms{D}$-module along a smooth divisor. Roughly speaking, this attempts to capture the notion of ``order of vanishing'' in purely $\ms{D}$-module-theoretic terms.

Let $X$ be a smooth variety, $D \subseteq X$ a smooth divisor and $\mc{M}$ a holonomic left $\ms{D}_X$-module. For convenience, we fix a local equation $t = 0$ for $D$ and a vector field $\partial_t$ such that $[\partial_t, t] = 1$.

\begin{definition} \label{defn:V-filtration}
The \emph{$V$-filtration of $\ms{D}_X$} is the decreasing $\mb{Z}$-indexed filtration given by
\[ V^n \ms{D}_X = \{ P \in \ms{D}_X \mid P t^m \in (t^{m + n}) \text{ for all }m \geq 0\}.\]
A \emph{$V$-filtration on $\mc{M}$} is an $\mb{R}$-indexed filtration $\{V^\alpha\mc{M} \subseteq \mc{M} \mid \alpha \in \mb{R}\}$ such that
\begin{enumerate}[label=(\Roman*)]
\item \label{itm:V-filtration 1} $V^\bullet \mc{M}$ is exhaustive, decreasing and left continuous (i.e.\ $V^{\alpha - \epsilon}\mc{M} = V^\alpha\mc{M}$),
\item \label{itm:V-filtration 2} the set $\{\alpha \in \mb{R} \mid V^\alpha \mc{M} \neq V^{>\alpha}\mc{M}\}$ is discrete,
\item \label{itm:V-filtration 3} we have $V^n \ms{D}_X \cdot V^\alpha \mc{M} \subseteq V^{\alpha + n}\mc{M}$,
\item \label{itm:V-filtration 4} the operator $t : V^\alpha\mc{M} \to V^{\alpha + 1}\mc{M}$ is an isomorphism for $\alpha \gg 0$,
\item \label{itm:V-filtration 5} each $V^\alpha\mc{M}$ is a coherent $V^0\ms{D}_X$-module,
\item \label{itm:V-filtration 6} for all $\alpha \in \mb{R}$, the operator $\alpha - \partial_t t : \gr_V^\alpha\mc{M} \to \gr_V^\alpha \mc{M}$ is nilpotent.
\end{enumerate}
\end{definition}

Note that, given the last condition, the remaining conditions are equivalent to requiring that $V^\bullet \mc{M}$ be decreasing, exhaustive and finitely generated over $V^\bullet \ms{D}_X$. This definition agrees with the one in \cite{MPVfiltration} and differs from the one in \cite{DV23}, for example, by a shift by $1$ in indexing. Intuitively, $V^\alpha \mc{M}$ can be thought of as the set of sections vanishing to order at least $\alpha - 1$ along $D$.

The $V$-filtration is unique if it exists, which it always does when $\mc{M}$ underlies a mixed Hodge module. An $\mb{R}$-indexed $V$-filtration as in Definition \ref{defn:V-filtration} need not exist for a general holonomic $\ms{D}_X$-module (although there is always a $\mb{C}$-indexed one \cite{Sabbah1987}).

\subsection{Extension of the Hodge filtration}\label{sec: extension of Hodge filtration}

The $V$-filtration is used to define the extension of the Hodge filtration of a mixed Hodge module across a principal divisor. We recall how this works in this subsection.

First, recall that any mixed Hodge module has an underlying filtered left $\ms{D}_X$-module. By convention, we will denote the $\ms{D}_X$-module by the same letter as the mixed Hodge module, and the Hodge filtration by $F_\bullet$.

Let $X$ be a smooth variety, $D \subseteq X$ a divisor, $U = X \setminus D$, and $j : U \to X$ the inclusion. The six functor formalism for mixed Hodge modules provides two extension functors
\[ j_!, j_* : \mhm(U) \to \mhm(X),\]
which are respectively the left and right adjoint to the restriction functor. Since we work with algebraic $\ms{D}$-modules, the functor $j_*$ is given simply by sheaf-theoretic pushforward at the level of $\ms{D}$-modules. This is not the correct functor at the level of filtered $\ms{D}$-modules: if, for $\mc{N} \in \mhm(U)$, we set
\[ F_p^{\mathit{naive}}j_*\mc{N} = j_*(F_p\mc{N}),\]
then $F_p^{\mathit{naive}}j_*\mc{N}$ need not be coherent over $\mc{O}_X$, since sections are allowed to have poles of arbitrary order along $D$. To rectify this, the theory places bounds on these poles using the $V$-filtration as follows.

Assume for simplicity that the divisor $D$ is smooth. Then we may form the $V$-filtration $V^\bullet j_*\mc{N}$ along $D$ as in the previous subsection. Intersecting with $V^\alpha j_*\mc{N}$ gives a filtration $F_\bullet^{\mathit{naive}}V^\alpha \mc{N}$. The Hodge filtration $F_\bullet j_*\mc{N}$ is defined by
\[ F_pj_*\mc{N} = \sum_k F_{p - k}\ms{D}_X \cdot (F_k^{\mathit{naive}}V^0j_*\mc{N}).\]
It follows (see \cite[Proposition 6.14.2]{MHMproject}, c.f. \cite[3.2.2]{Saito88}) that
\begin{equation} \label{eq:hodge equals naive}
F_p V^\alpha j_*\mc{N} = F_p^{\mathit{naive}}V^\alpha j_*\mc{N} \quad \text{if $\alpha \geq 0$}.
\end{equation}
For the dual extension $j_!\mc{N}$, one can show that the natural map $j_!\mc{N} \to j_*\mc{N}$ induces an isomorphism
\[ V^{>0} j_!\mc{N} \overset{\sim}\to V^{>0}j_*\mc{N}\]
and thus a filtration $F_\bullet^{\mathit{naive}}V^{>0}j_!\mc{N}$. The Hodge filtration on $j_!\mc{N}$ is defined similarly by
\[ F_p j_!\mc{N} = \sum_k F_{p - k} \ms{D}_X \cdot (F_k^{\mathit{naive}} V^{>0}j_!\mc{N}).\]

\subsection{The main theorem}\label{sec: Msfs construction}

Fix as before a smooth complex variety $X$ and a (possibly singular) divisor $D \subseteq X$ with complement $U = X \setminus D$ and inclusion $j : U \to X$. We assume that $D$ is given by a local equation $f : X \to \mb{C}$ with graph embedding $\iota : X \to X \times \mb{C}_t$ sending $x$ to $(x, f(x))$. We fix a mixed Hodge module $\mc{M}$ on $X$ on which $f$ acts bijectively, or equivalently, so that $\mc{M} = j_*\mc{N}$ for some $\mc{N} \in \mhm(U)$.

Our aim is to write an expression for the $V$-filtration on $\iota_+\mc{M}$ along the smooth divisor $\{t = 0\}$ in terms of the Hodge filtrations on extension functors. The starting point is the following result of Malgrange (due in this form to Musta\c{t}\u{a} and Popa).

Consider the $\ms{D}_X[s]$-module $\mc{M}[s]f^s$ given by $\mc{M}[s]f^s=\mc{M}\otimes \C[s]$ as an $\cO_X[s]$-module, with $\ms{D}_X$-action defined by
\[ \xi \cdot (us^jf^s)=\left(\xi(u)s^j+\frac{\xi(f)}{f}us^{j+1}\right)f^s \]
for a vector field $\xi \in T_X$ and $u \in \mc{M}$. The formula makes sense since $f$ acts invertibly on $\mc{M}$. We equip $\mc{M}[s]f^s$ with the structure of a $\ms{D}_{X \times \mb{C}^*} = \ms{D}_X\langle t, t^{-1}, \partial_t t\rangle$-module by setting $s = -\partial_t t$ and
\[ t \cdot (us^jf^s) = u (s + 1)^j f^{s + 1} = fu (s + 1)^j f^s.\]

\begin{proposition}{\cite[Proposition 2.5]{MPVfiltration}}\label{prop: Malgrange isomorphism}
The morphism
    \begin{align}\label{eqn: Malgrange isomorphism}
    \mc{M}[s]f^s &\xrightarrow{\sim} \iota_{+}\mc{M}, \quad us^j f^s \mapsto u\otimes (-\d_tt)^j,
    \end{align}
is an isomorphism of $\sD_{X \times \mb{C}^*}$-modules with the inverse
    \begin{equation}\label{eqn: inverse Malgrange isomorphism}
    u\otimes \d_t^j \mapsto \frac{u}{f^j}\prod_{i=0}^{j-1}(-s+i)f^s.
    \end{equation}
\end{proposition}
Here we have used the explicit description of the graph embedding as 
\[ \iota_+ \mc{M} = \bigoplus_{k \geq 0} \mc{M} \otimes \partial_t^k.\]

Now consider the restrictions to $U$. Denote by $\iota_U:U\to U\times \C$ the restriction of the graph embedding. Endow $\mc{N}[s]f^s$ and $\iota_{U,+}\mc{N}$ with the filtrations
\[ F_p \mc{N}[s]f^s = \sum_{j + k = p} F_j\mc{N} s^k f^s \quad \text{and} \quad  F_p \iota_{U,+}\mc{N} = \sum_{j + k = p - 1} F_j \mc{N} \otimes \partial_t^k.\]
The latter is simply the Hodge filtration for the mixed Hodge module $\iota_{U,+}\mc{N}$ on $U \times \mb{C}$.

\begin{lemma} \label{lem:filtered Malgrange}
The isomorphism of Proposition \ref{prop: Malgrange isomorphism} defines a filtered isomorphism
\[ (\mc{N}[s]f^s, F_\bullet) \cong (\iota_{U,+} \mc{N}, F_{\bullet + 1}).\]
\end{lemma}
\begin{proof}
This follows directly from the construction since $1/f$ is a regular function on $U$.
\end{proof}

Now, the key observation is that the (non-holonomic) filtered $\ms{D}_U$-module $\mc{N}[s]f^s$ can be approximated by mixed Hodge modules as follows. For any $\alpha \in \mb{R}$, we can write
\[ \mc{N} f^{-\alpha} \colonequals \frac{\mc{N}[s]f^s}{(s + \alpha)} =  \mc{N} \otimes f^*\left(\mc{O}_{\mb{C}^*}t^{-\alpha}\right).\]
Since $\alpha \in \mb{R}$, $\mc{O}_{\mb{C}^*}t^{-\alpha}$ is a unitary local system on $\mb{C}^*$, hence a complex Hodge module \cite[Theorem 16.2.1]{MHMproject}, so this defines a mixed Hodge module structure on $\mc{N}f^{-\alpha}$. More generally, for any $n > 0$, the $\ms{D}_{\mb{C}^*}$-module $\mc{O}_{\mb{C}^*}[s]t^s/(s + \alpha)^n$
underlies an admissible variation of mixed Hodge structure, with Hodge and weight filtrations given by
\[ F_p\left(\frac{\mc{O}_{\mb{C}^*}[s]t^s}{(s + \alpha)^n}\right) = \sum_{j \leq p} \frac{\mc{O}_{\mb{C}^*}s^jt^s}{(s + \alpha)^n} \quad \text{and} \quad W_w\left(\frac{\mc{O}_{\mb{C}^*}[s]t^s}{(s + \alpha)^n}\right) = \sum_{2k \geq -w}\frac{\mc{O}_{\mb{C}^*}(s + \alpha)^kt^s}{(s + \alpha)^n}.\]
This is the standard ``nilpotent orbit'' tensored with the unitary local system $\mc{O}_{\mb{C}^*}t^{-\alpha}$. So we have a natural mixed Hodge module structure on
\[ \frac{\mc{N}[s]f^s}{(s + \alpha)^n} = \mc{N} \otimes f^*\left(\frac{\mc{O}_{\mb{C}^*}[s]t^s}{(s + \alpha)^n}\right).\]
Note that this comes equipped with a (nilpotent) morphism of mixed Hodge modules
\[ s + \alpha : \frac{\mc{N}[s]f^s}{(s + \alpha)^n}(1) \to \frac{\mc{N}[s]f^s}{(s + \alpha)^n}\]
whose cokernel is $\mc{N}[s]f^s/(s + \alpha)^{n - 1}$. Moreover, since the Hodge filtration on $\mc{N}$ is bounded below, we have
\begin{equation} \label{eq:basic limit}
F_p \left(\frac{\mc{N}[s]f^s}{(s + \alpha)^n}\right) = F_p\mc{N}[s]f^s \quad \text{for $n \gg 0$}.
\end{equation}
Taking the union over all $p$, we recover the entire filtered $\ms{D}_U[s]$-module $\mc{N}[s]f^s$ from the inverse system of mixed Hodge modules $\{\mc{N}[s]f^s/(s + \alpha)^n\}_{n \geq 0}$.

This so far seems fairly trivial. However, applying the functors $j_!$ and $j_*$, we get something quite non-trivial. Following \cite[\S 3.3]{DV23}, we define
\begin{align*}
F_pj_{\ast}^{(-\alpha)}\cN[s]f^s\colonequals \varprojlim_n F_p j_{\ast}\left(\frac{\cN[s]f^s}{(s + \alpha)^n}\right), &\quad j_{\ast}^{(-\alpha)}\cN[s]f^s\colonequals \bigcup_p F_pj_{\ast}^{(-\alpha)}\cN[s]f^s,\\
F_pj_{!}^{(-\alpha)}\cN[s]f^s\colonequals \varprojlim_n F_p j_{!}\left(\frac{\cN[s]f^s}{(s+\alpha)^n}\right), &\quad j_{!}^{(-\alpha)}\cN[s]f^s\colonequals \bigcup_p F_pj_{!}^{(-\alpha)}\cN[s]f^s.
\end{align*}
By \eqref{eq:basic limit}, the inverse limits are constant for $n \gg 0$. Hence, each $j_!^{(-\alpha)}\mc{N}[s]f^s$ and $j_*^{(-\alpha)}\mc{N}[s]f^s$ is a quasi-coherent filtered $\ms{D}_X[s]$-module whose restriction to $U$ is $(\mc{N}[s]f^s, F_\bullet)$. In particular, we have tautological filtered morphisms
\begin{equation} \label{eq:tautological morphisms}
(j_!^{(-\alpha)}\mc{N}[s]f^s, F_\bullet) \to (j_*^{(-\alpha)}\mc{N}[s]f^s, F_\bullet) \to (j_*\mc{N}[s]f^s, j_*F_\bullet) = (\iota_+\mc{M}, F_{\bullet + 1}^{\mathit{naive}}),
\end{equation}
where the last equality is Proposition \ref{prop: Malgrange isomorphism} and Lemma \ref{lem:filtered Malgrange}.

\begin{thm}\label{thm: main thm in the text}
The morphisms \eqref{eq:tautological morphisms} are strict and injective, and induce identifications
\[ j_!^{(-\alpha)}\mc{N}[s]f^s \cong V^{> \alpha} \iota_+\mc{M} \quad \text{and} \quad j_*^{(-\alpha)}\mc{N}[s]f^s \cong V^\alpha \iota_+\mc{M} \]
for all $\alpha \in \mb{R}$.
\end{thm}

Since $F_\bullet^{\mathit{naive}}V^\alpha\iota_+\mc{M} = F_\bullet V^\alpha \iota_+\mc{M}$ for $\alpha \geq 0$ by \eqref{eq:hodge equals naive}, this implies Theorem \ref{thm: V filtration via power of fs}.

\begin{proof}
It was shown in \cite[Lemma 3.5]{DV23} that the morphisms \eqref{eq:tautological morphisms} are strict and injective and that
\[ U^\alpha  \iota_+\mc{M} \colonequals j_*^{(-\alpha)}\mc{N}[s]f^s \subseteq \iota_+\mc{M}\]
is a decreasing, exhaustive, left-continuous filtration such that
\[ U^{>\alpha} \iota_+\mc{M} = j_!^{(-\alpha)}\mc{N}[s]f^s.\]
The proof of this is not difficult: it reduces to a straightforward calculation in the case when $D$ is smooth using the definition of $j_!$ and $j_*$ in terms of the $V$-filtration. It therefore remains to check that $U^\alpha \iota_+\mc{M}$ satisfies conditions \ref{itm:V-filtration 2}-\ref{itm:V-filtration 6} of Definition \ref{defn:V-filtration}.

For \ref{itm:V-filtration 2}, by construction, $ j_!^{(-\alpha)}\mc{N}[s]f^s \neq  j_*^{(-\alpha)}\mc{N}[s]f^s$ if and only if $j_!\mc{N}f^{-\alpha} \neq j_*\mc{N}f^{-\alpha}$. This is a discrete set by, for example, \cite[Proposition 3.3]{DV23}.

For \ref{itm:V-filtration 3} and \ref{itm:V-filtration 4}, we note that the isomorphism $t : \mc{N}[s]f^s \cong \mc{N}[s]f^s$ descends to an isomorphism of mixed Hodge modules $\mc{N}[s]f^s/(s + \alpha)^n \cong \mc{N}[s]f^s/(s + \alpha + 1)^n$ for all $n$. It follows that
\[ t : U^\alpha \iota_+\mc{M} \overset{\sim}\to U^{\alpha + 1}\iota_+\mc{M} \quad \text{for all $\alpha$},\]
proving \ref{itm:V-filtration 4}. Since $U^\alpha \iota_+\mc{M} \subseteq \iota_+\mc{M}$ is a $\ms{D}_X[s]$-submodule by construction and $V^\bullet\ms{D}_{X \times \mb{C}}$ is generated over $\ms{D}_X$ by $t$ in degree $1$ and $-\partial_t t = s$ in degree $0$, this also implies \ref{itm:V-filtration 3}.

For \ref{itm:V-filtration 5}, we need to show that $j_{\ast}^{(-\alpha)}\cN[s]f^s$ is a \emph{coherent} $\sD_X\langle s,t\rangle$-module. In fact, we show that it is a coherent $\sD_X[s]$-module. Recall that $(j_{\ast}(\cN\cdot f^{-\alpha}),F_{\bullet})$ is a coherent filtered $\sD_X$-module so that we may choose local generators $u_i\in F_{p_i}j_{\ast}(\cN\cdot  f^{-\alpha})$ for $1\leq  i \leq r$. The map
\[ ev_{s=-\alpha}:j_{\ast}^{(-\alpha)}\cN[s]f^s \to j_{\ast}(\cN\cdot f^{-\alpha})\]
is filtered surjective; let us choose pre-images $\tilde{u}_i \in F_{p_i}j_{\ast}^{(-\alpha)}\cN[s]f^s$ of the $u_i$. Set $F_k(\sD_X[s])=\sum_{p+q\leq k}F_p\sD_X s^q$. We prove by induction on $n$ that the morphism
\[ \bigoplus_i F_{p-p_i}\frac{\sD_X[s]}{(s+\alpha)^n}\cdot \tilde{u}_i \to F_pj_{\ast}\frac{\cN[s]f^s}{(s+\alpha)^n}\]
is surjective for all $n$ and all $p$. Since $F_pj_*^{(-\alpha)}\mc{N}[s]f^s = F_pj_*\mc{N}[s]f^s/(s +\alpha)^n$ for $n \gg 0$, this implies that
\[ \bigoplus_i F_{p-p_i}(\sD_X[s])\cdot \tilde{u}_i \to F_pj_{\ast}^{(-\alpha)}\cN[s]f^s\]
is surjective for all $p$, and hence that the $\tilde{u}_i$ generate $ j_{\ast}^{(-\alpha)}\cN[s]f^s$ as a filtered $\sD_X[s]$-module.  This gives the coherence over $\sD_X[s]$.

The base case is $n=1$, which is true by construction. For $n>1$, we have the following commutative diagram
\[ \begin{tikzcd}
0\arrow[r] &\oplus_i F_{p-p_i-(n-1)}\sD_X u_i \arrow[r,"(s+\alpha)^{n-1}"] \arrow[d]& \oplus_i F_{p-p_i}\frac{\sD_X[s]}{(s+\alpha)^n} \tilde{u}_i \arrow[r] \arrow[d]&\oplus_i F_{p-p_i}\frac{\sD_X[s]}{(s+\alpha)^{n-1}}\tilde{u}_i \arrow[d]\arrow[r] &0\\
0\arrow[r] & F_{p-(n-1)}j_{\ast}(\cN \cdot f^{\alpha}) \arrow[r,"(s+\alpha)^{n-1}"] & F_pj_{\ast}\frac{\cN[s]f^s}{(s+\alpha)^n}\arrow[r] &F_pj_{\ast}\frac{\cN[s]f^s}{(s+\alpha)^{n-1}} \arrow[r] &0
\end{tikzcd}\]
with exact rows. By construction and induction, the first and third column are surjective, respectively. Therefore by the five lemma, the second morphism is also surjective.

Finally, to prove \ref{itm:V-filtration 6}, recall that by the existence of the $b$-function \cite[\S 2]{Beilinson} (see also \cite[\S 6]{DV22}),
\[ \mc{C}_n \colonequals \coker\left(j_!\left(\frac{\mc{N}[s]f^s}{(s + \alpha)^n}\right) \to j_*\left(\frac{\mc{N}[s]f^s}{(s + \alpha)^n}\right)\right) \]
$\mc{C}_n$ stabilises for $n \gg 0$. Furthermore, by (the proof of) \cite[Lemma 3.6]{DV23}, there exists $N\in \Z_{\geq 0}$ so that $(s+\alpha)^N$ annihilates $\mc{C}_n$ for all $n$. By the eventual constancy of the inverse systems defining $F_pj_!^{(-\alpha)}\mc{N}[s]f^s$ and $F_pj_*^{(-\alpha)}\mc{N}[s]f^s$ for fixed $p$ (which follows from \eqref{eq:basic limit}), and the fact that any morphism of mixed Hodge modules is strict with respect to the Hodge filtration, one has
\[ \gr_U^\alpha\iota_+\mc{M} = \coker\left(j_!^{(-\alpha)}\mc{N}[s]f^s \to j_*^{(-\alpha)}\mc{N}[s]f^s\right) = \mc{C}_n, \quad \textrm{for $n \gg 0$.} \]
 It follows that $(s + \alpha)^N$ annihilates $\gr_U^\alpha\iota_+\mc{M} $, so this proves \ref{itm:V-filtration 6}.
\end{proof}

\section{Applications to higher multiplier and Hodge ideals}\label{sec: higher multiplier and Hodge ideal}

In this section, we give an illustrative application of our main theorem to the theory of higher multiplier and Hodge ideals by proving a precise comparison between the two. We also take the opportunity to explain how these fit neatly into the framework of complex Hodge modules, and to clarify some points in the literature.

Throughout the section, we fix a smooth complex variety $X$ and an effective divisor $D \subseteq X$ and write $j : U = X \setminus D \hookrightarrow X$ for the inclusion of the complement. We assume for simplicity that $D$ is reduced and given as the divisor of a regular function $f : X \to \mb{C}$. 

\subsection{Higher multiplier and Hodge ideals}

The higher multiplier and Hodge ideals are defined as follows.

Denote by $\iota:X\to X\times \C$ the graph embedding of $f$. Then, for $k \in \mb{Z}_{\geq 0}$ and $\alpha \in \Q$, the associated \emph{higher multiplier ideal} $\tilde{I}_{k}(\alpha D)\colonequals \cI_{k,-\alpha}(D)$  is defined in \cite[Definition 3.2]{SY23} (which is equivalent to the microlocal multiplier ideals in \cite{Saito16}):
\begin{equation}\label{definition: higher multiplier ideals}
\tilde{I}_{k}(\alpha D)\otimes \partial_t^k=\gr^F_{k+1}V^{\alpha}\iota_{+}\cO_X \subseteq \gr^F_{k + 1}\iota_+\mc{O}_X = \mc{O}_X \otimes \partial_t^k,
\end{equation}
where
\[ F_{k + 1}\iota_{+}\cO_X=\sum_{0\leq \ell\leq k} \cO_X\otimes \d_t^\ell, \quad F_{k + 1}V^{\alpha}\iota_{+}\cO_X=F_{k + 1}\iota_{+}\cO_X\cap V^{\alpha}\iota_{+}\cO_X,\]
It is manifest from this definition that $\tilde{I}_{k}(\alpha D)$ is a sheaf of ideals.

Similarly, when $D$ is reduced, the \emph{Hodge ideal} $I_k(\alpha D)$ is defined by
\begin{equation} \label{eq:hodge ideal}
I_k(\alpha D) \otimes \mc{O}_X(kD)f^{-\alpha} = F_k j_*(\mc{O}_U f^{-\alpha}) =: F_k \mc{O}_X(*D) f^{-\alpha} \subseteq \mc{O}_X(*D)f^{-\alpha}
\end{equation}
for $k \in \mb{Z}_{\geq 0}$ and $\alpha \in \mb{Q}_{> 0}$. One way to see that $I_k(\alpha D)$ is indeed an ideal is to note that, under the isomorphism of Proposition \ref{prop: Malgrange isomorphism}, we have
\begin{equation} \label{eq:graph hodge in s coords}
F_{k + 1}\iota_+\mc{O}_X\overset{\sim}\to \sum_{\ell = 0}^k \mc{O}_X(\ell D) \binom{s}{\ell}f^s.
\end{equation}
Since $\alpha > 0$, by \cite[Lemma 3.1.7]{Saito88} $V^\alpha \iota_+\mc{O}_X(*D) = V^\alpha \iota_+\mc{O}_X \subseteq \iota_+\mc{O}_X$. So Corollary \ref{corollary: ses of Hodge on V} implies that
\[
F_k \mc{O}_X(*D)f^{-\alpha} = \mathit{ev}_{s = -\alpha}(F_{k + 1}V^\alpha \iota_+\mc{O}_X(*D)) \subseteq \mathit{ev}_{s = -\alpha}(F_{k + 1} \iota_+\mc{O}_X) = \mc{O}_X(k D) f^{-\alpha}\]
and hence $I_k(\alpha D) \subseteq \mc{O}_X$ is an ideal. This formula for the Hodge filtration on $\mc{O}_X(*D)f^{-\alpha}$ is originally due to Musta\c{t}\u{a} and Popa \cite[Theorem A]{MPVfiltration}.

Strictly speaking, \eqref{eq:hodge ideal} is different from the original definition of the Hodge ideals, which did not use the language of complex mixed Hodge modules. The two definitions nevertheless agree.

\begin{lemma}\label{lemma: two definitions of Hodge ideals agree}
Definition \eqref{eq:hodge ideal} agrees with the definition of Hodge ideals in \cite[\S 4]{MPbirational}.
\end{lemma}
\begin{proof}
The Hodge ideal in \cite{MPbirational} is defined by
\[ I_k(\alpha D) \otimes \mc{O}_X(kD)f^{-\alpha} = F_k'\mc{O}_X(*D) f^{-\alpha},\]
where the filtration $F_\bullet'$ is defined as follows. Write $\alpha = \frac{m}{n}$ for $m, n \in \mb{Z}_{> 0}$. Extracting an $n$th root of $f$, we get an $n$-fold cyclic covering $\pi : \tilde X \to X$ with Galois group $\mu_n$ so that
\begin{equation} \label{eq:MP vs complex}
\mc{O}_X(*D)f^{-\alpha} \subseteq \pi_*\mc{O}_{\tilde X}(*\pi^{-1}(D))
\end{equation}
is the summand on which $\zeta \in \mu_n$ acts by $\zeta^{-m}$. The $\ms{D}$-module $\pi_*\mc{O}_{\tilde X}(*\pi^{-1}(D))$ underlies the (rational) mixed Hodge module $\pi_*\tilde{j}_*\mb{Q}^H [\dim X] = j_*\pi^\circ_*\mb{Q}^H[\dim X]$, where $\tilde{j} : \pi^{-1}(U) \to \tilde X$ is the inclusion and $\pi^\circ : \pi^{-1}(U) \to U$ is the restriction of $\pi$, fitting into the following diagram.
\[\begin{tikzcd} \pi^{-1}(U) \arrow[d,"\pi^\circ"] \arrow[r,"\tilde{j}"] & \tilde{X}\arrow[d,"\pi"]\\
U \arrow[r,"j"] & X\end{tikzcd}\]
Hence $\pi_*\mc{O}_{\tilde X}(*\pi^{-1}(D))$ comes with a Hodge filtration $F_\bullet$. 
The filtration $F_\bullet'$ is defined in \cite{MPbirational} as the restriction of this Hodge filtration to $\mc{O}_X(*D)f^{-\alpha}$ via \eqref{eq:MP vs complex}. But \eqref{eq:MP vs complex} is an inclusion of complex mixed Hodge modules (it is an inclusion of mixed Hodge modules on $U$, and thus is after applying $j_{\ast}$), so the filtration $F_{\bullet}'$ on $\mc{O}_X(*D)f^{-\alpha}$ agrees with the Hodge filtration $F_\bullet$ on $\mc{O}_X(*D)f^{-\alpha}$ in \eqref{eq:hodge ideal}, and hence the two definitions of Hodge ideals coincide.
\end{proof}

\begin{remark}
For $D$ reduced, the weighted Hodge ideals $I_k^{W_{\bullet}}(D)$, as a filtration of the Hodge ideal $I_k(D)$, have been introduced in \cite{Olano}. By Lemma \ref{lemma: two definitions of Hodge ideals agree}, we can use \eqref{eq:hodge ideal} to define weighted multiplier ideals for arbitrary $D$ and $\alpha\in \Q_{>0}$ by
\[ W_{\ell}I_k(\alpha D)\otimes \cO_X(kD)\cdot f^{-\alpha}=F_kW_{\dim X+\ell}j_{\ast}(\cO_Uf^{-\alpha}), \quad \forall \ell\geq 0.\]
One can prove that they generalise many properties of $I_k^{W_{\bullet}}(D)$ as in \cite{Olano}. However we have chosen not to include them in this note. For a more systematic treatment from this perspective, see \cite{dakin}. 
\end{remark}

\subsection{Comparison between higher multiplier and Hodge ideals} 

We now give the proof of Theorem \ref{thm:higher vs hodge} that the higher multiplier ideals can be recovered from the family of Hodge ideals by taking a limit. The proof is based on the observation above that both ideals can be written in terms of the same sheaf $F_{k + 1}V^\alpha \iota_+\mc{O}_X$.

We begin with some basic lemmas about families of ideals over curves. Recall from the introduction that if $X$ and $S$ are schemes over $\mb{C}$, then a \emph{family of ideals} over $S$ is a sheaf of ideals $\mc{I} \subseteq \mc{O}_{X \times S}$ such that the quotient $\mc{O}_{X \times S}/\mc{I}$ is flat over $S$. When $X$ is projective, this is the same data as a morphism from $S$ to the Hilbert scheme $\mrm{Hilb}(X)$.

\begin{lem} \label{lem:unique limit}
Let $X$ be a scheme over $\mb{C}$, $S$ a smooth curve and $U \subseteq S$ a dense open subset. Then any family of ideals $\mc{I}_U$ in $\mc{O}_X$ over $U$ extends uniquely to a family over $S$.
\end{lem}

\begin{proof}
We may assume without loss of generality that $X = \spec A$, $S = \spec R$ and $U = \spec R'$ are affine. Note also that since $S$ is a smooth curve, a module over $R$ (or its localisation $R'$) is flat if and only if it is torsion-free; in particular every submodule of a flat module is flat. Now, to prove existence, let us write
\[ \mc{I} = \mc{I}_U \cap (A \otimes R) \subseteq (A \otimes R),\]
where we regard $\mc{I}_U$ as an ideal in $A \otimes R' \supseteq A \otimes R$. By construction, $(A \otimes R)/\mc{I} \to (A \otimes R')/\mc{I}_U$ is injective. Since $(A \otimes R')/\mc{I}_U$ is flat over $R'$, it is flat over $R$ and hence so is its submodule $(A \otimes R)/\mc{I}$. So $\mc{I}$ defines an extension to a family of ideals over $S$ as desired. To prove uniqueness, suppose that $\mc{I}'$ is another such extension of $\mc{I}_U$. Then clearly $\mc{I}' \subseteq \mc{I}$. So
\[ \mc{I}/\mc{I}' \subseteq (A \otimes R)/\mc{I}' \]
is an $R$-submodule supported on the complement of $U$. Since the $R$-module $(A\otimes R)/\mc{I}'$ is torsion-free, we therefore have $\mc{I}/\mc{I}' = 0$, so $\mc{I} = \mc{I}'$.
\end{proof}

\begin{lem} \label{lem:unique extension}
Let $X$ be a scheme of finite type over $\mb{C}$, $S$ a smooth curve and $Z \subseteq S$ a Zariski-dense subset. If $\mc{I}$ and $\mc{I}'$ are families of ideals in $\mc{O}_X$ over $S$ such that $\mc{I}_s = \mc{I}'_s$ for all $s \in Z$, then $\mc{I} = \mc{I}'$.
\end{lem}
\begin{proof}
We may assume without loss of generality that $X = \spec A$ is affine. Consider the quotients $\mc{O}_{X \times S}/\mc{I}$ and $\mc{O}_{X \times S}/\mc{I}'$. By generic freeness, there exists a dense open subset $U \subseteq S$ such that $\mc{O}_{X \times S}/\mc{I}|_U$ and $\mc{O}_{X \times S}/\mc{I}'|_U$ are free $\mc{O}_U$-modules. We claim that the set
\[ W = \{s \in U \mid \mc{I}'_{s} \subseteq \mc{I}_{s} \} \]
is closed in $U$. Indeed, $W$ is the vanishing locus of the map of free $\mc{O}_U$-modules
\[ \mc{O}_{X \times S}/\mc{I} |_U \to \mc{O}_{X \times S}/\mc{I}'|_U, \]
which, as the intersection of the vanishing loci of the matrix coefficients, is manifestly closed. Since $Z \cap U \subseteq W$ and $Z$ is Zariski-dense, we conclude that $W = U$, i.e., that $\mc{I}'|_U \subseteq \mc{I}|_U$. Interchanging the roles of $\mc{I}$ and $\mc{I}'$, we deduce that $\mc{I}|_U = \mc{I}'|_U$. We conclude that $\mc{I} = \mc{I}'$ by applying Lemma \ref{lem:unique limit}.
\end{proof}

We will also use the following flatness criterion.

\begin{lem} \label{lem:flatness criterion}
Let $X$ be a scheme over $\mb{C}$, $S$ a smooth curve and $\mc{I} \subseteq \mc{O}_{X \times S}$ a sheaf of ideals. Then $\mc{O}_{X \times S}/\mc{I}$ is $S$-flat at $s \in S$ if and only if the morphism $\mc{I}|_{X \times \{s\}} \to \mc{O}_{X \times \{s\}}$ is injective.
\end{lem}
\begin{proof}
The sheaf $\mc{O}_{X \times S}/\mc{I}$ is $S$-flat at $s$ if and only if $\mrm{Tor}_i^{\mc{O}_S}(\mc{O}_s, \mc{O}_{X \times S}/\mc{I}) = 0$ for $i > 0$. Since $S$ is a smooth curve and $\mc{I}$ is $\mc{O}_S$-torsion-free, $\mc{I}$ itself is flat over $S$. The lemma now follows from the long exact sequence associated to $0 \to \mc{I} \to \mc{O}_{X \times S} \to \mc{O}_{X \times S}/\mc{I} \to 0$.
\end{proof}

\begin{proof}[Proof of Theorem \ref{thm:higher vs hodge}]
Uniqueness follows from Lemma \ref{lem:unique extension} since every interval in $\mb{R}$ is Zariski-dense in $\mb{P}^1$. We next prove existence. To simplify the notation, write
\[ V^\gamma := V^\gamma \iota_+\mc{O}_X(*D) \quad \text{for $\gamma \in \mb{R}$}.\]
Since $\alpha > 0$, $V^\alpha \subseteq \iota_+\mc{O}_X$, so from the formula \eqref{eq:graph hodge in s coords}, we have
\begin{equation} \label{eq:higher vs hodge 1}
 F_{k + 1}V^\alpha \subseteq F_{k + 1}\iota_+\mc{O}_X \subseteq \mc{O}_X(k D) \otimes \mb{C}[s]_{\leq k} f^s,
\end{equation}
where $\mb{C}[s]_{\leq k}$ is the space of polynomials in $s$ of degree at most $k$. The idea of the construction is to define a family of ideals by taking the image of this subspace under evaluation at different values of $s$. More precisely, note that
\[ \mb{C}[s]_{\leq k} = H^0(\mb{P}^1, \mc{O}(k \cdot \infty)),\]
where we identify $-s$ with the coordinate $z$ on $\mb{P}^1$. So \eqref{eq:higher vs hodge 1} provides a morphism
\begin{equation} \label{eq:higher vs hodge 2}
 F_{k + 1} V^\alpha \otimes_{\mc{O}_X} \mc{O}_{X \times \mb{P}^1}(- k(D \times \mb{P}^1) - k (X \times \infty)) \to \mc{O}_{X \times \mb{P}^1}.
 \end{equation}
Let $\mc{I} \subseteq \mc{O}_{X \times \mb{P}^1}$ be the image of \eqref{eq:higher vs hodge 2}. By generic flatness, there exists a dense open $U \subseteq \mb{P}^1$ over which $\mc{O}_{X \times \mb{P}^1}/\mc{I}$ is flat. We let $\mathfrak{I}_{k,\alpha}$ be the unique extension of $\mc{I}|_{X \times U} $ to $X \times \mb{P}^1$ provided by Lemma \ref{lem:unique limit}.

We next check that the restriction of $\mathfrak{I}_{k,\alpha}$ to $z = \beta$ is $I_k(\beta D)$ for $\beta = \alpha - \epsilon$, $0 \leq \epsilon \ll 1$: note that $V^\beta = V^\alpha$ in this range. By construction, 
\begin{equation} \label{eq:higher vs hodge 3}
\mc{I}|_{X \times \{\beta\}} = \frac{F_{k + 1}V^\beta}{F_{k + 1}V^\beta \cap (s + \beta)\mb{C}[s]\cdot F_{k + 1}V^\beta} \otimes \mc{O}_X(-kD).
\end{equation}
Now, by Corollary \ref{corollary: ses of Hodge on V}, the morphism $s + \beta \colon V^\beta \to V^\beta\{-1\}$ is strict, so
\[ F_{k + 1}V^\beta \cap (s + \beta) V^\beta= (s + \beta)F_k V^\beta. \]
Hence, the denominator of \eqref{eq:higher vs hodge 3} satisfies
\[
(s + \beta)F_k V^\beta \subseteq F_{k + 1}V^\beta \cap (s + \beta) \mb{C}[s]\cdot F_{k + 1}V^\beta \subseteq F_{k + 1}V^\beta\cap (s + \beta) V^\beta = (s + \beta)F_kV^\beta.
\]
So by Corollary \ref{corollary: ses of Hodge on V}, \eqref{eq:higher vs hodge 3} becomes
\[ \mc{I}|_{X \times \{\beta\}} = \frac{F_{k + 1}V^\beta}{(s + \beta) F_k V^\beta} \otimes \mc{O}_X(-k D) = I_k(\beta D).\] 
Since $I_k(\beta D)$ injects into $\mc{O}_X$, we conclude that $\mc{O}_{X \times \mb{P}^1}/\mc{I}'$ is flat at $z = \beta$ by Lemma \ref{lem:flatness criterion}, and hence that
\[ \mathfrak{I}_{k,\alpha}|_{X \times \{\beta\}} = \mc{I}|_{X \times \{\beta\}} = I_k(\beta D) \quad \text{for $\beta = \alpha - \epsilon$, $0 \leq \epsilon \ll 1$}.\]

Finally, we show that the fibre of $\mathfrak{I}_{k,\alpha}$ at $z = \infty$ coincides with the higher multiplier ideal $\tilde{I}_k(\alpha D)$. We have
\begin{equation} \label{eq:higher vs hodge 5}
\mc{I}|_{X \times \{\infty\}} = \frac{F_{k + 1}V^\alpha}{F_{k + 1}V^\alpha \cap s^{-1}\mb{C}[s^{-1}]\cdot F_{k + 1}V^\alpha} \otimes \mc{O}_X(-kD)s^{-k}.
\end{equation}
Since $\alpha > 0$,
\[ F_{k + 1}V^\alpha = V^\alpha \cap \mc{O}_X(*D) \otimes \mb{C}[s]_{\leq k}f^s.\]
So
\[ F_{k + 1}V^\alpha \cap s^{-1}\mb{C}[s^{-1}]\cdot F_{k + 1}V^\alpha \subseteq V^\alpha \cap \mc{O}_X(*D) \otimes \mb{C}[s]_{\leq k -1}= F_kV^\alpha. \]
Conversely, $s F_k V^\alpha \subseteq F_{k + 1}V^\alpha$, so
\[ F_{k + 1}V^\alpha \cap s^{-1}\mb{C}[s^{-1}]\cdot F_{k + 1}V^\alpha =  F_kV^\alpha.\]
Plugging into \eqref{eq:higher vs hodge 5}, we get
\[ \mc{I}|_{X \times \{\infty\}} = \gr^F_{k + 1} V^\alpha \otimes \mc{O}_X(-k D) s^{-k} = \gr^F_{k + 1} V^\alpha \otimes \partial_t^{-k} = \tilde{I}_k(\alpha D)\]
Since $\tilde{I}_k(\alpha D)$ injects into $\mc{O}_X$, we conclude as above that $\mathfrak{I}_{k,\alpha}|_{X \times \{\infty\}} = \mc{I}|_{X \times \{\infty\}} = \tilde{I}_k(\alpha D)$, as claimed.
\end{proof}

\section{Vanishing theorems for Hodge ideals via twisted mixed Hodge modules} \label{vanishing theorem}

In contrast to the setting in \S \ref{sec: higher multiplier and Hodge ideal}, let us now drop the assumption that $D$ is given as the divisor of a global function $f : X \to \mb{C}$. In this setting, Hodge ideals do not come directly from the Hodge filtrations on any mixed Hodge modules. They do, however, come from Hodge filtrations on \emph{twisted} mixed Hodge modules. As this theory is not as well-known to singularity theorists as it could be, we include a brief discussion below. As an application, we give quick proofs for their vanishing theorems.

For any line bundle $\mc{L}$ on $X$ and any $\lambda \in \mb{C}$, there is a sheaf of rings $\ms{D}_{X, \lambda\mc{L}}$ on $X$, called the sheaf of \emph{$\lambda\mc{L}$-twisted differential operators}, constructed as follows. First, form the $\mb{G}_m$-torsor
\[ p : \mrm{Tot}(\mc{L})^\times = \spec_X\left(\bigoplus_{n \in \mb{Z}} \mc{L}^{\otimes n}\right) \to X,\]
with $\mb{G}_m$-action defined so that $\mc{L}^{\otimes n}$ has degree $n$. The sheaf of $\mb{G}_m$-invariant differential operators
\[ \widetilde{\ms{D}}_{X, \mc{L}} := p_*(\ms{D}_{\mrm{Tot}(\mc{L})^\times})^{\mb{G}_m}\]
is a quasi-coherent sheaf of algebras on $X$, with centre $\mb{C}[\theta]$ generated by the derivative $\theta$ of the $\mb{G}_m$-action (acting on $\mc{L}^{\otimes n}$ by multiplication by $n$). The sheaf of twisted differential operators is defined by
\[ \ms{D}_{X, \lambda \mc{L}} \colonequals \widetilde{\ms{D}}_{X, \mc{L}} \otimes_{\mb{C}[\theta], \lambda} \mb{C} =\widetilde{\ms{D}}_{X, \mc{L}}/(\theta -\lambda).\]
Any trivialisation of $\mc{L}$ determines an isomorphism $\widetilde{\ms{D}}_{X, \mc{L}}= \ms{D}_X[\theta]$ and hence $\ms{D}_{X, \lambda \mc{L}} \cong \ms{D}_X$. In particular, $\ms{D}_{X, \lambda \mc{L}}$ is locally isomorphic to the sheaf of differential operators, but not necessarily globally. The definition is cooked up so that $\mc{L}$ itself is canonically a $\ms{D}_{X, \mc{L}}$-module, even though it need not admit any structure of a $\ms{D}_X$-module.

One can define $\lambda\mc{L}$-twisted Hodge modules by copying the definition of ordinary complex Hodge modules with $\ms{D}_{X, \lambda\mc{L}}$ in place of $\ms{D}_X$. This works because $\ms{D}_{X, \lambda\mc{L}}$ is locally isomorphic to $\ms{D}_X$ and the definition of Hodge module is local. See, for example, \cite[\S 2]{SY23} for the polarised case. Alternatively, one can construct the twisted theory inside the untwisted theory as follows.

By descent, the category $\qcoh(\widetilde{\ms{D}}_{X, \mc{L}})$ of quasi-coherent $\widetilde{\ms{D}}_{X, \mc{L}}$-modules is equivalent to the category $\qcoh^{\mb{G}_m}(\ms{D}_{\mrm{Tot}(\mc{L})^\times})$ of quasi-coherent $\mb{G}_m$-equivariant modules over the $\mb{G}_m$-equivariant sheaf of rings $\ms{D}_{\mrm{Tot}(\mc{L})^\times}$ (this is sometimes called ``weak'' equivariance, as it is weaker than the usual notion of equivariance for $\ms{D}$-modules). The functor in one direction is pullback $p^*$ and in the other is $p_*(-)^{\mb{G}_m}$. A less obvious fact (see, e.g., \cite[Lemma 2.5.4]{BBjantzen}) is that the pullback functor
\begin{equation} \label{eq:twisted pullback}
 p^* : \qcoh(\ms{D}_{X, \lambda\mc{L}}) \subseteq \qcoh(\widetilde{\ms{D}}_{X, \mc{L}}) \to \qcoh(\ms{D}_{\mrm{Tot}(\mc{L})^\times})
\end{equation}
is also fully faithful, without any equivariance on the target. In other words, for fixed $\lambda$, there is at most one $\mb{G}_m$-action on a $\ms{D}_{\mrm{Tot}(\mc{L})^\times}$-module $\mc{M}$ such that $p_*(\mc{M})^{\mb{G}_m}$ is a $\mc{D}_{X, \lambda \mc{L}}$-module.

This motivates the definition of twisted mixed Hodge module below.

\begin{definition}
A \emph{$\lambda\mc{L}$-twisted mixed Hodge module on $X$} is a mixed Hodge module on $\mrm{Tot}(\mc{L})^\times$ whose underlying $\ms{D}_{\mrm{Tot}(\mc{L})^\times}$-module lies in the image of \eqref{eq:twisted pullback}. We write $\mhm(\ms{D}_{X, \lambda\mc{L}})$ for the category of such things.
\end{definition}

This is the definition used in \cite{DV23}, for example. The category $\mhm(\ms{D}_{X, \lambda\mc{L}})$ is zero unless $\lambda \in \mb{R}$ (e.g., \cite[Proposition 2.6]{DV23}). By \cite[Proposition 2.9]{DV23}, for example, the Hodge filtration on the $\ms{D}_{\mrm{Tot}(\mc{L})^\times}$-module underlying an object in $\mhm(\ms{D}_{X, \lambda\mc{L}})$ is preserved by the $\mb{G}_m$-action, and hence descends to a Hodge filtration $F_\bullet \mc{M}$ on the associated $\ms{D}_{X, \lambda\mc{L}}$-module. 

Now let us relate twisted mixed Hodge modules to global Hodge ideals. Let $D$ be an effective divisor on $X$ with a section $f\in H^0(X,\cL)$ where $\cL=\cO_X(D)$. Let $\cM$ be a mixed Hodge module where local equations of $D$ act bijectively. The tautological section of $\cO_X(D)$ defines a function
\[ f : \mrm{Tot}(\mc{L})^\times \to \mb{C}\]
on which $\mb{G}_m$ (and hence $\theta$) acts with weight $1$. For $\alpha \in \mb{R}$, the tautological $\mb{G}_m$-action on the $\ms{D}_{\mrm{Tot}(\mc{L})^\times}$-module $p^*\mc{M}$ gives a weak $\mb{G}_m$-action on $p^*\mc{M} \cdot f^{\beta}$ such that $\theta$ acts on
\[ \mc{M} \cdot f^{\beta} \colonequals p_*(p^*\mc{M} \cdot f^{\beta})^{\mb{G}_m} \]
by $\beta$. In other words, $\mc{M} \cdot f^{\beta}$ is a $\ms{D}_{X,\beta\mc{L}}$-module. Note that any local equation $g$ for $D$ determines a trivialisation $\mc{L} \cong \mc{O}_X$ sending $f$ to $g$ and the $\ms{D}_{X, \beta\mc{L}}$-module $\mc{M}\cdot f^{\beta}$ to the $\ms{D}_X$-module $\mc{M} \cdot g^{\beta}$, so this agrees with the construction in \cite[\S 2]{SY23}. Since $p^*\mc{M} \cdot f^{\beta}$ underlies a mixed Hodge module on $\mrm{Tot}(\mc{L})^\times$, we deduce:
\begin{proposition} \label{prop: twisting by a line bundle}
$\cM\cdot f^{\beta}$ underlies a $\beta \mc{L}$-twisted mixed Hodge module.
\end{proposition}

In particular, we have an associated Hodge filtration
\[ F_\bullet(\mc{M} \cdot f^{\beta}) \colonequals p_*(F_\bullet(p^*\mc{M} \cdot f^{\beta}))^{\mb{G}_m}.\]
It follows that the Hodge ideal $I_k(\alpha D)$ can be defined by a global version of \eqref{eqn: new point of view of Hodge ideals}:
    \[ I_k(\alpha D)\otimes \cO_X(kD)\cong I_k(\alpha D)\otimes \cO_X(kD)\cdot f^{-\alpha}\colonequals F_k(\cO_X(\ast D)\cdot f^{-\alpha}), \quad \forall \alpha>0,\]
where in the first isomorphism we identify $\cO_X(\ast D)\cdot f^{-\alpha}$ with $\cO_X(\ast D)$ as a quasi-coherent sheaf on $X$ by ignoring $f^{-\alpha}$.

As mentioned in the introduction, in \cite{Chenbingyi}, Bingyi Chen introduces \[ \cM_k(\alpha D)= \frac{\omega_X(kD)\otimes I_k(\alpha D)}{\omega_X((k-1)D)\otimes I_{k-1}(\alpha D)},\]
and a Spencer complex $\mathrm{Sp}_k(\cM_{\bullet}(\alpha D))$. It is direct to see that 
\[
\cM_k(\alpha D) =\omega_X\otimes \gr^F_k(\cO_X(\ast D)\cdot f^{-\alpha}),\quad \mathrm{Sp}_k(\cM_{\bullet}(\alpha D))=\gr^F_k\DR(\cO_X(\ast D)\cdot f^{-\alpha}),
\]
where $\gr^F_k\DR(\cM)$ stands for the graded piece of the de Rham complex of a twisted $\sD$-module  (see e.g. \cite[\S 2.7]{SY23}). 
\begin{proof}[Proof of Theorem \ref{thm: Bingyi Chen}]
Since $\cO_X(\ast D)\cdot f^{-\alpha}\otimes \cO_X(B)$ is $(B-\alpha L)$-twisted by Proposition \ref{prop: twisting by a line bundle}, what we want follows from the twisted Kodaira vanishing \cite[Theorem 7.2]{DV23} (see also \cite[Theorem 1.5]{SY23}). 
\end{proof}

\begin{proof}[Proof of Theorem \ref{thm: strengthen of Bingyi}]
The first case follows by the same argument as the proof of \cite[Theorem 1.7]{SY23}. The second case is a special case of \cite[Theorem 1.3]{DV23} (when $X$ is a full flag variety) or \cite[Theorem 1.11]{DMB} (when $X$ is a partial flag variety). The third case follows from Theorem \ref{thm: Bingyi Chen} using the Euler sequence as in the case of untwisted Hodge modules \cite[Remark 2.34, 4)]{Saito90}.
\end{proof}

\section{Functoriality for higher multiplier ideals}\label{sec: functoriality higher multiplier ideal}

In this section, we observe that the functoriality for mixed Hodge modules gives, in principle, a method for computing higher multiplier ideals from a log resolution. This resolves some questions of Schnell and the second-named author (see Question 10.1 and Problem 10.2 of the first \texttt{arXiv} version of \cite{SY23}).

Recall that if $\pi : Y \to X$ is a projective morphism, then the direct image of a mixed Hodge module $\mc{M}$ on $Y$ is given at the level of filtered $\ms{D}$-modules by Laumon's formula
\begin{equation} \label{eq:laumon}
(\pi_*\mc{M}, F_\bullet) \colonequals \mbf{R}\pi_*((\ms{D}_{X \leftarrow Y}, F_\bullet) \Lotimes_{(\ms{D}_Y, F_\bullet)} (\mc{M}, F_\bullet)),
\end{equation}
as an object in the filtered derived category of $\ms{D}_X$-modules. Here we give the bimodule $\ms{D}_{Y \leftarrow X}$ the filtration by order of differential operator starting in degree $\dim X - \dim Y$. It is an important theorem of Saito (see \cite[Th\'eor\`eme 1]{Saito88}, \cite[Theorem 2.14]{Saito90} and \cite[Theorem 14.3.1]{MHMproject}), essential for the construction of the six functor formalism, that this filtered complex is \emph{strict}, i.e.\ the associated spectral sequence degenerates at the first page.

Now suppose that $\pi : \tilde X \to X$ is a log resolution of $(X, D)$, which is an isomorphism over $X\setminus D$, and write $\tilde D = \pi^*(D)$. Then $\mc{O}_X(*D)f^{-\alpha} = \pi_*\mc{O}_{\tilde X} (*\tilde D)f^{-\alpha}$ as mixed Hodge modules. So \eqref{eq:laumon} gives
\begin{equation} \label{eq:laumon for hodge ideals}
(\mc{O}_X(*D)f^{-\alpha}, F_\bullet) = \mbf{R}\pi_*((\ms{D}_{X \leftarrow \tilde X}, F_\bullet) \Lotimes_{(\ms{D}_{\tilde X}, F_\bullet)} (\mc{O}_{\tilde X}(*\tilde D)f^{-\alpha}, F_\bullet)).
\end{equation}
In \cite[\S 6]{MPbirational}, Musta\c{t}\u{a} and Popa use the fact that $(\tilde X, \tilde D)$ has simple normal crossings to write down an explicit resolution of $(\mc{O}_{\tilde X}(* \tilde D)f^{-\alpha}, F_\bullet)$ by locally free filtered $\ms{D}_{\tilde X}$-modules. Combining with \eqref{eq:laumon for hodge ideals}, this gives an explicit complex computing the Hodge ideals $I_k(\alpha D)$ \cite[Theorem 8.1]{MPbirational}.

In \cite{SY23}, it was asked by Schnell and the second-named author whether one has a similar birational description of the higher multiplier ideals $\tilde{I}_k(\alpha D)$. We note here that this is indeed the case. The main ingredient is the following.

\begin{prop} \label{prop:bistrict pushforward}
Let $\pi : Y \to X$ be a projective morphism and $\mc{M} \in \mhm_{\bar{\mb{Q}}}(Y)$. Abusing notation, write $\iota$ for the graph embeddings of both $f$ and $f \circ \pi$. Then the complex
\[ \mc{C}^{\bullet}\colonequals \mbf{R}\pi_*((\ms{D}_{X \leftarrow Y}, F_\bullet) \Lotimes_{(\ms{D}_Y, F_\bullet)} (V^\bullet \iota_{+}\mc{M}, F_\bullet)) \]
is bi-strict with respect to the filtrations $F_\bullet$ and $V^\bullet$, and in particular
\[ (V^\alpha\iota_+\mc{H}^i(\pi_*\mc{M}), F_\bullet) = \mc{H}^i(\mbf{R}\pi_*((\ms{D}_{X \leftarrow Y}, F_\bullet) \Lotimes_{(\ms{D}_Y, F_\bullet)} (V^\alpha \iota_{+}\mc{M}, F_\bullet)))\]
for all $i$ and all $\alpha\geq 0$.
\end{prop}

Proposition \ref{prop:bistrict pushforward} is a special case of \cite[Proposition 3.3.17]{Saito88}. Saito proves this statement for any $\alpha$  assuming only that the filtered $\ms{D}_{Y \times \mb{C}}$-module $(\iota_+\mc{M}, F_\bullet)$ is quasi-unipotent and regular along $Y \times \{0\}$ and that
\[ \mbf{R}\pi_*((\ms{D}_{X \leftarrow Y}, F_\bullet) \Lotimes_{(\ms{D}_Y, F_\bullet)} (\gr_V^\alpha\mc{M}, F_\bullet)) \]
is strict for all $\alpha$. The first of these conditions holds by definition for any mixed Hodge module and the second by the direct image theorem. One can also deduce Proposition \ref{prop:bistrict pushforward} from Theorem \ref{thm: V filtration via power of fs} as follows (although this argument is strictly speaking circular, as \cite[Proposition 3.3.17]{Saito88} is used to set up the theory of mixed Hodge modules in the first place). This perspective will be useful in \cite{DY}.

\begin{proof}
The basic idea is that Theorem \ref{thm: V filtration via power of fs} allows us to reduce the bi-strictness to the strictness of \eqref{eq:laumon}. Recall from \cite[Remark 10.2.5]{MHMproject} that the bi-strictness is equivalent to all canonical morphisms
\[ \cH^i(F_p\mc{C}^{\bullet})\to F_p\cH^i(\mc{C}^{\bullet}), \quad \cH^i(V_{\alpha}\mc{C}^{\bullet})\to V_{\alpha}\cH^i(\mc{C}^{\bullet}),\]
\begin{equation}\label{eqn: bistrictness} \cH^i\left(F_pV^{\alpha}\mc{C}^{\bullet}\right)\to F_pV^{\alpha}\cH^i(\mc{C}^{\bullet}) \end{equation}
being isomorphisms for $p,i\in \Z$ and $\alpha\geq 0$.

It is not hard to show that 
\[ \mathbf{R}\pi_*\left((\ms{D}_{X \leftarrow Y}, F_\bullet) \Lotimes_{(\ms{D}_Y, F_\bullet)} (\iota_{+}\mc{M}, F_\bullet)\right)=\mathbf{R}(\pi\times \textrm{id}_{\C})_*\left((\ms{D}_{X\times \C \leftarrow Y\times \C}, F_\bullet) \Lotimes_{(\ms{D}_{Y\times \C}, F_\bullet)} (\iota_{+}\mc{M}, F_\bullet)\right)\]
Since $\iota_{+}\cM\in \MHM(Y\times \C)$, the first property follows from the strictness of \eqref{eq:laumon} for $(\pi\times \textrm{id}_{\C})_{\ast}(\iota_{+}\cM)$. The second property is equivalent to the fact that the following canonical morphism is an isomorphism:
\[ \cH^i\left(\pi_{+}(V^{\alpha}\iota_{+}\cM)\right)\to V^{\alpha}\cH^i\left((\pi\times \textrm{id}_{\C})_{+}\iota_{+}\cM)\right).\]
This is classical, see e.g.\ \cite[Theorem 7.5.2]{SabbahnoteDmodule}. It remains to prove that \eqref{eqn: bistrictness} is an isomorphism.

Fix $\alpha\geq 0$ and $p\in \Z$. Set $D=f^{-1}(0)$ and consider the following commutative diagram:
\[\begin{tikzcd}
{} &U_Y\colonequals \pi^{-1}(U) \arrow[dl,"f\circ \pi_U",swap] \arrow[d,"\pi_U"] \arrow[r,"j_Y"] & Y \arrow[d,"\pi"]\\
\C^{\ast} &\arrow[l,"f",swap] U=X- D \arrow[r,"j"] & X
\end{tikzcd}\]
By Theorem \ref{thm: V filtration via power of fs} and the eventual constancy of the inverse systems in Theorem \ref{thm: V filtration via power of fs} for fixed $p$ (see \S \ref{sec: Msfs construction}), one can choose $n\gg 0$ so that 
\begin{align*}
F_pV^{\alpha}\iota_{+}(\cH^i\pi_{\ast}\cM)=F_pj_{\ast}\left(\frac{j^{\ast}\cH^i\pi_{\ast}\cM[s]f^s}{(s+\alpha)^n}\right).
\end{align*}
Since $\cH^i(\mc{C}^{\bullet})=\cH^i(\iota_{+}\pi_{\ast}\cM)=\iota_{+}(\cH^i\pi_{\ast}\cM)$ as mixed Hodge modules, we have
\begin{align*}
F_pV^{\alpha}\cH^i(\mc{C}^{\bullet})&=F_pV^{\alpha}\iota_{+}(\cH^i\pi_{\ast}\cM)\\
&=F_pj_{\ast}\left(j^{\ast}\cH^i\pi_{\ast}\cM\otimes f^{\ast}\frac{\cO_{\C^{\ast}}[s]t^s}{(s+\alpha)^n}\right)\\
&=F_pj_{\ast}\left(\cH^i\pi_{U,\ast}(j_Y^{\ast}\cM)\otimes f^{\ast}\frac{\cO_{\C^{\ast}}[s]t^s}{(s+\alpha)^n}\right)\\
&=F_pj_{\ast}\left(\cH^i\pi_{U,\ast}\left(j_Y^{\ast}\cM\otimes (f\circ \pi_U)^{\ast}\frac{\cO_{\C^{\ast}}[s]t^s}{(s+\alpha)^n}\right)\right)\\
&=F_p\cH^i\pi_{\ast}\cM_{\alpha},
\end{align*}
where we set 
\[ \cM_{\alpha} \colonequals j_{Y,\ast}\left(j_Y^{\ast}\cM\otimes (f\circ \pi_U)^{\ast}\frac{\cO_{\C^{\ast}}[s]t^s}{(s+\alpha)^n}\right)\in \MHM(Y).\]
In above, we have used the proper base change formula, projection formula and the fact that $j_{Y,\ast}$ is an exact functor in the category of mixed Hodge modules. 

On the other hand, applying Theorem \ref{thm: V filtration via power of fs} again, we can let $n\gg 0$ so that
\[ F_qV^{\alpha}\iota_{+}\cM=F_q\cM_{\alpha}, \quad \textrm{for all $q\leq p$}.\]
Note that for any filtered $\sD_Y$-module $(\cN,F_{\bullet})$, $F_p\pi_{\ast}\cN$ only involves $F_q\cN$ for $q\leq p$,  it follows that
\[ \cH^i(F_pV^{\alpha}\mc{C}^{\bullet})=\cH^i\left(F_p\pi_{\ast}(V^{\alpha}\iota_{+}\cM)\right)=\cH^i\left(F_p\pi_{\ast}\cM_{\alpha}\right).\]
Furthermore, by chasing the isomorphisms above, one can identify the canonical morphism \eqref{eqn: bistrictness} with the canonical morphism
\[ \cH^i\left(F_p\pi_{\ast}\cM_{\alpha}\right)\to F_p\cH^i\left(\pi_{\ast}\cM_{\alpha}\right). \]
Since this is an isomorphism by the strictness of \eqref{eq:laumon}, we finish the proof.
\end{proof}

\begin{proof}[Proof of Proposition \ref{prop: containment of higher multiplier ideal}]
Consider a log resolution $\pi : \tilde X \to X$ of $(X, D)$ which is an isomorphism over $X\setminus D$. Set $\tilde{D}=\pi^{\ast}(D)$, then as above $\pi_{\ast}\mc{O}_{\tilde X}(*\tilde D)=\cO_X(\ast D)$.  Since $V^{>0}\iota_{+}\cO_X=V^{>0}\iota_{+}\cO_X(\ast D)$ (similarly for $\cO_{\tilde{X}}$ and $\cO_{\tilde{X}}(\ast \tilde{D}))$, applying Proposition \ref{prop:bistrict pushforward} to $\mc{O}_{\tilde X}(*\tilde D)$ we get
\[ (V^\alpha\iota_+\mc{O}_X, F_\bullet) = \mbf{R}\pi_*((\ms{D}_{X \leftarrow \tilde X}, F_\bullet) \otimes_{(\ms{D}_{\tilde X}, F_\bullet)} (V^\alpha \iota_+\mc{O}_{\tilde X}, F_\bullet)), \quad \textrm{for $\alpha > 0$}. \]
Taking associated gradeds, we get a graded isomorphism
\begin{equation} \label{eq:laumon for grFV}
\bigoplus_k \gr^F_{k + 1} V^\alpha\iota_+\mc{O}_X = \mbf{R}\pi_*\left(\omega_{\tilde X/X} \otimes \mrm{Sym}(\pi^*T_X) \Lotimes_{\mrm{Sym}(T_{\tilde X})}\left(\bigoplus_k \gr^F_{k + 1} V^\alpha\iota_+\mc{O}_{\tilde X}\right)\right),
\end{equation}
where $\pi^*T_X$ and $T_{\tilde X}$ have graded degree $1$. It follows that the tautological morphism $\mrm{Sym}(T_{\tilde X}) \to \mrm{Sym}(\pi^*T_X)$ of $\mrm{Sym}(T_{\tilde X})$-modules defines a morphism
\begin{align*}
\bigoplus_k \pi_*(\omega_{\tilde X/X} \otimes \tilde{I}_{k}(\alpha\tilde D)) &\otimes \partial_t^k = \pi_*\left(\omega_{\tilde X/X} \otimes \mrm{Sym}(T_{\tilde X}) \Lotimes_{\mrm{Sym}(T_{\tilde X})}\left(\bigoplus_k \tilde{I}_{k}(\alpha\tilde D) \otimes \partial_t^k\right)\right), \\
&\to \mbf{R}\pi_*\left(\omega_{\tilde X/X} \otimes \mrm{Sym}(\pi^*T_{X}) \Lotimes_{\mrm{Sym}(T_{\tilde X})}\left(\bigoplus_k \tilde{I}_{k}(\alpha\tilde D) \otimes \partial_t^k\right)\right) \\
&= \bigoplus_k \tilde{I}_{k}(\alpha D) \otimes \partial_t^k,
\end{align*}
which is the identity outside $D$. This gives the desired inclusion of ideals for all $k$.
\end{proof}

\bibliographystyle{alpha}
\bibliography{reference}{}

\vspace{\baselineskip}

\footnotesize{
\textsc{School of Mathematics and Statistics, University of Melbourne, Parkville, VIC, 3010, Australia} \\
\indent \textit{E-mail address:} \href{mailto:dougal.davis1@unimelb.edu.au}{dougal.davis1@unimelb.edu.au}

\vspace{\baselineskip}

\textsc{Department of Mathematics, University of Kansas, 1450 Jayhawk Blvd, Lawrence, KS 66045, United States} \\
\indent \textit{E-mail address:} \href{mailto:ruijie.yang@ku.edu}{ruijie.yang@ku.edu} 
}
\end{document}